\documentclass{amsart}
\usepackage{amssymb,amsmath,amscd,xy,graphicx,textcomp}
\newtheorem{theorem}{Theorem}[section]
\newtheorem{lemma}[theorem]{Lemma}

\newtheorem{definition}[theorem]{Definition}

\newtheorem{proposition}[theorem]{Proposition}

\xyoption{arrow}

\xyoption{matrix}

\def\C{\mathbb{C}}
\def\R{\mathbb{R}}
\def\Z{\mathbb{Z}}

\def\tree{\mathcal{T}}

\def\ql{\backslash \! \backslash}

\title{Semigroups of $sl_3(\C)$ tensor product invariants}
\author{Christopher Manon ${}^{\dagger}$, Zhengyuan Zhou ${}^{\ddagger}$}
\thanks{The first author was supported by the NSF fellowship DMS-0902710\\
$\dagger$ UC Berkeley, George Mason University\\
$\ddagger$ UC Berkeley, Stanford University}

\begin{document}

\begin{abstract} 
We find presentations for a family of semigroup algebras related to the problem of decomposing $sl_3(\C)$ tensor products.  Along the way we find new toric degenerations 
of the Grassmannian variety $Gr_3(\C^n)$ which are $T-$invariant for $T \subset GL_n(\C)$ the diagonal torus.  
\end{abstract}

\maketitle

\tableofcontents

\smallskip

\section{Introduction}

A central problem in combinatorial representation theory is to compute the dimension
of the space of invariant vectors 

\begin{equation}
H(\vec{\lambda}) = (V(\lambda_1) \otimes \ldots \otimes V(\lambda_n))^{G}\\
\end{equation}

\noindent
 in a tensor product of irreducible representations of a reductive group $G$, or equivalently a reductive Lie algebra $Lie(G) = \mathfrak{g}.$   This problem is related to many interesting results each with their own combinatorial flavor, such as the Littlewood-Richardson rule, Kostant's multiplicity formula, Steinberg's formula, Littelman's path model, the quiver methods of Derksen and Weyman \cite{DW}, and the crystal bases of Kashiwara and Lusztig \cite{Lu}.   Interest in formulas for the dimensions of these spaces also stems from their numerous appearances in geometry, the Schubert calculus in the cohomology of flag varieties, and the Hilbert functions of classical configuration spaces from invariant theory to name two.  See Kumar's survey paper \cite{K} for an account of many of these results.

It is an especially beautiful feature of reductive Lie groups that this dimension can always be found by counting the lattice points in a convex polytope which depends on the weights $\lambda_1, \ldots, \lambda_n$ and some extra combinatorial information, namely a trivalent tree $\tree$ with $n$ leaves labelled $\{1, \ldots, n\}$ see \cite{BZ1}, \cite{BZ2}, \cite{GP},\cite{KTW}, \cite{M}.  Let $\Delta$ be a Weyl chamber of $\mathfrak{g},$ the following is proved in \cite{M}.

\begin{theorem}
For each trivalent tree $\tree$ with $n$ labelled leaves, there is a cone $Q_{\tree}(\mathfrak{g})$, along with a linear map $\pi_{\tree}: Q_{\tree}(\mathfrak{g}) \to \Delta^n,$
such that the number of lattice points in $\pi_{\tree}^{-1}(\vec{\lambda}) = Q_{\tree}(\vec{\lambda})$ is equal to the dimension of $H(\vec{\lambda}).$
\end{theorem}

Many nonisomorphic cones which serve in this role can be constructed, but we will gloss over this point as only one type, see section \ref{quilt}, will be important here. The object $Q_{\tree}(\mathfrak{g})$ has essentially been available at least since the paper of Berenstein and Zelevinksy, \cite{BZ1}.  In the case $n =3,$ the cone $Q_3(GL_m(\C))$ 
was the subject of much work of Knutson, Tao, and Woodward, \cite{KTW}, wherein the saturation conjecture of Klyachko was resolved.  The $sl_2(\C)$ cases $Q_{\tree}(sl_2(\C))$ have been understood in general, mostly in their role as toric degenerations of the Grassmannian variety $Gr_2(\C^n)$, see \cite{SpSt} and our remarks below. However, not much has been said about the structure of these semigroups when $n > 3,$ $\mathfrak{g} \neq sl_2(\C)$.  In this paper we study the "next" case, we determine the algebraic structure of the semigroups $Q_{\tree}(sl_3(\C))$ for trees of arbitrary size and topology.

\begin{theorem}\label{main}
The semigroup $Q_{\tree}(sl_3(\C))$ is generated by lattice points $w$ such that
$\pi_{\tree}(w)$ are all fundamental weights of $sl_3(\C).$  These generators are 
always subject to at most quadratic and cubic relations.  
\end{theorem}

A rather beautiful feature of the semigroup $Q_{\tree}(sl_2(\C))$ is that its generators
can be put in bijection with the paths in $\tree$ that connect distinct leaves, see \cite{SpSt}. We find a similar result holds for each $Q_{\tree}(sl_3(\C)).$  We say a subtree $\tree' \subset \tree$ is proper if every leaf of $\tree'$ is a leaf of $\tree.$

\begin{theorem}\label{main2}
There are exactly $2$ generators of $Q_{\tree}(sl_3(\C))$ for each proper $\tree' \subset \tree.$  Each semigroup $Q_{\tree}(sl_3(\C))$ is minimally generated by $2(2^n - n -1)$ elements. 
\end{theorem}

\noindent
We also find a combinatorial description of the defining relations of $Q_{\tree}(sl_3(\C)).$

Now for a few more words about the relation of the spaces $H(\vec{\lambda})$ to invariant theory. For general reductive $G$, the spaces $H(\vec{\lambda})$ serve as the global sections of an ample line bundle $\mathcal{L}_{\vec{\lambda}}$ on a $configuration$ $space,$

\begin{equation}
M_{\vec{\lambda}} = G \ql_{\vec{\lambda}} [G/P_1 \times \ldots \times G/P_n].\\
\end{equation}

\noindent
Here $P_i$ is the parabolic subgroup of $G$ associated to $\lambda_i$, $G/P_i$ is the flag variety, and the quotient is a Mumford Geometric Invariant Theory quotient.   When $G = SL_m(\C)$, and each $\lambda_i$ is rank $1,$ ie a multiple of the first fundamental weight $\omega_1,$ the space $M_{\vec{\lambda}}$ is a space of weighted configurations of $n$ points on the projecive space $\mathbb{P}^m.$  

Determining the structure of the projective coordinate ring $R_{\vec{\lambda}}$ of this space is a classical problem in invariant theory, see for example \cite{HMSV} and \cite{D}, $11.2$.

In \cite{M}, the first author constructed toric degenerations of $R_{\vec{\lambda}}$ to sub-algebras of $\C[Q_{\tree}(\mathfrak{g})],$ for all $\vec{\lambda}.$  We hope that by understanding the basic structure of $Q_{\tree}(sl_3(\C))$, the necessary combinatorics to resolve this question can be brought into reach for $sl_3(\C)$.  A similar program was successful in resolving the problem for points on $\mathbb{P}^1$ in part by understanding the cone $Q_{\tree}(sl_2(\C)),$ see \cite{HMSV}. 

We now recall the first fundamental theorem of invariant theory for $sl_3(\C),$ stated
in a slightly peculiar way.

\begin{theorem}[First Fundamental Theorem of Invariant Theory]
The projective coordinate ring $P_{3, n}$ of the Pl\"ucker embedding of the Grassmannian variety $Gr_3(\C^n)$ can be identified with the $sl_3(\C)$ invariant vectors in the algebra $\C[x_1, x_2, x_3]^{\otimes n}.$  
\end{theorem}

In terms of representation theory, this says that the Pl\"ucker algebra $P_{3, n}$ is isomorphic to the sum $\bigoplus_{\vec{r} \in \Z_{\geq 0}^n} (V(r_1\omega_1) \otimes \ldots \otimes V(r_n\omega_1))^{sl_3(\C)}.$  Results of \cite{M} imply that this algebra flatly degenerates to the semigroup algebra $\C[P_{\tree}]$ where $P_{\tree} \subset Q_{\tree}(sl_3(\C))$ is the subsemigroup of lattice points $w$ such that $\pi(w) \in \Delta_{sl_3(\C)}^n$ is always a tuple of multiples of the first fundamental weight $\omega_1.$  In this sense, the semigroup algebras $\C[P_{\tree}]$ are an analogue of the toric degenerations of $Gr_2(\C^n)$ constructed by Sturmfels and Speyer in \cite{SpSt}. 

\begin{equation}
\pi(w) = (r_1\omega_1, \ldots, r_n\omega_1)\\
\end{equation}

We also determine the structure of this semigroup.  We state our answer in terms of proper subtrees of $\tree,$ as above. 

\begin{theorem}\label{grass}
The minimal generators of the semigroup $P_{\tree}$ are in bijection with proper subtrees  $\tree' \subset \tree$ which have the property that for any two leaves $\ell_1, \ell_2 \in \tree',$ the number of trinodes on the unique path between $\ell_1$ and $\ell_2$ is odd.    The relations among these generators are generated by quadratics and cubics. 
\end{theorem}

We will show that a degeneration constructed by Sturmfels and Miller, \cite{MSt}  of $Gr_3(\C^n),$ to a toric variety corresponding to a semigroup of Gel'fand-Tsetlin patterns is realized by $P_{\tree_0}$, where $\tree_0$ is a special tree.

\section{Description of $Q_{\tree}(sl_m(\C))$}\label{quilt}

In this section we describe the affine semigroups $Q_{\tree}(sl_m(\C)).$
Elements of these semigroups are built out of triangular diagrams called Berenstein-Zelevinksy triangles.  

\begin{definition}[Berenstein-Zelevinsky triangles]
Let $BZ_3(m)$ be the semigroup of non-negative integer weightings of the vertices of the version of the diagram below with $m-1$ small triangles to an edge, which satisfy the condition that any pair of pairs $A, B$ and $E, D$ on the opposite side of a hexagon in the middle of the diagram satisfies $A + B = E+ D.$
\end{definition}

\begin{figure}[htbp]
\centering
\includegraphics[scale = 0.35]{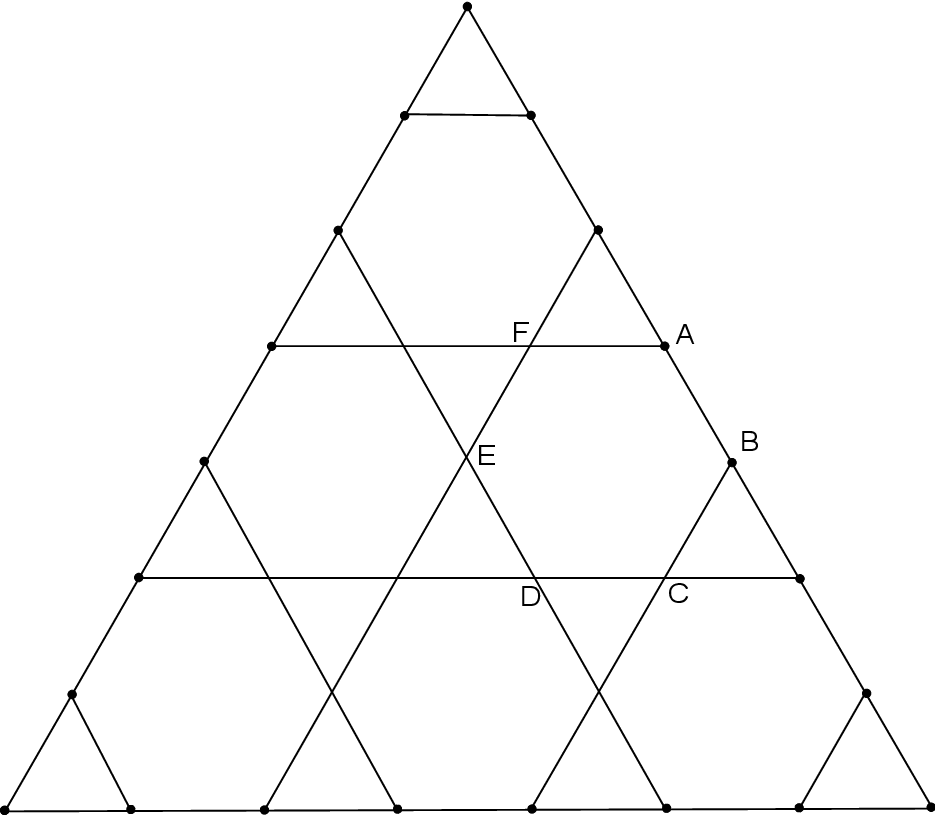}
\caption{The $BZ$ triangle diagram for $sl_5(\C)$}
\label{fig:BZ}
\end{figure}

We refer the reader to Figure \ref{fig:BZ} for an image of one of these patterns.   
 We orient the edges of the triangle counter-clockwise,  and label the entries along the three edges $(a_1, \ldots, a_{2(m-1)}), (b_1, \ldots, b_{2(m-1)}), (c_1, \ldots, c_{2(m-1)}).$   We define the boundary weights $\lambda_i$ as follows.

$$
\lambda_1 = (a_1+ a_2)\omega_1 + \ldots + (a_{2m-3} + a_{2m-2})\omega_{m-1}
$$

$$
\lambda_2 = (b_1+ b_2)\omega_1 + \ldots + (b_{2m-3} + b_{2m-2})\omega_{m-1}
$$

$$
\lambda_3 = (c_1+ c_2)\omega_1 + \ldots + (c_{2m-3} + c_{2m-2})\omega_{m-1}
$$

Here $\omega_i = (1, \ldots, 1, 0, \ldots, 0)$ is the $i-$th fundamental weight of $sl_m(\C).$ 

  We will utilize a dual description of this semigroup given by Knutson, Tao, and Woodward in terms of weighted graphs on the BZ triangle diagram called honeycombs.  These objects also appear in the work Gleizer and Postnikov on tensor product multiplicities, \cite{GP}.

\begin{figure}[htbp]
\centering
\includegraphics[scale = 0.4]{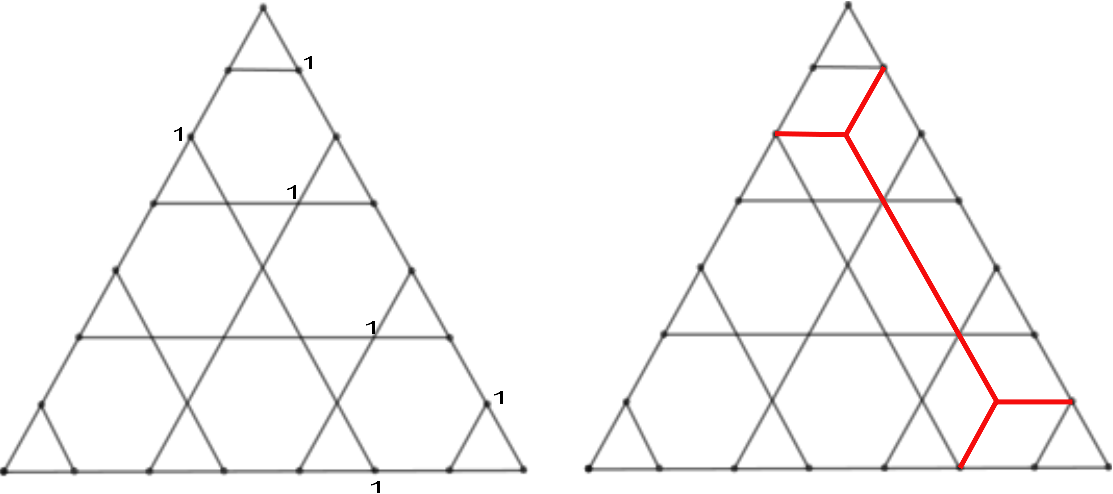}
\caption{A BZ triangle and its honeycomb graph. A lack of entry on a triangle indicates a $0$.}
\label{fig:HiveBZ}
\end{figure}

\noindent
A honeycomb is constructed from the BZ triangle by replacing each entry $a$ on a hexagon with an edge weighted $a$ connected to the center of the hexagon, see Figure \ref{fig:HiveBZ}.  A weight $b$ on a corner of the triangle is left to represent a vertex weighted by $b.$  The graphs coming from this contruction can be characterized as those graphs that are "balanced" at each vertex dual to the interior of a hexagon.  This means that the sum of the vectors pointing to the corners of the hexagon, scaled by the weights, is $0$.  Following the recipe given above, the BZ triangle/honeycomb in Figure \ref{fig:HiveBZ} has boundary weights $\lambda_1 = (1 + 0)\omega_1 + (0 + 1) \omega_4,$ $\lambda_2 = (1 +0) \omega_2,$  $\lambda_3 = (0 + 1) \omega_3$.  Addition is performed on two of these objects by taking their union and counting multiplicities. 

\begin{figure}[htbp]
\centering
\includegraphics[scale = 0.4]{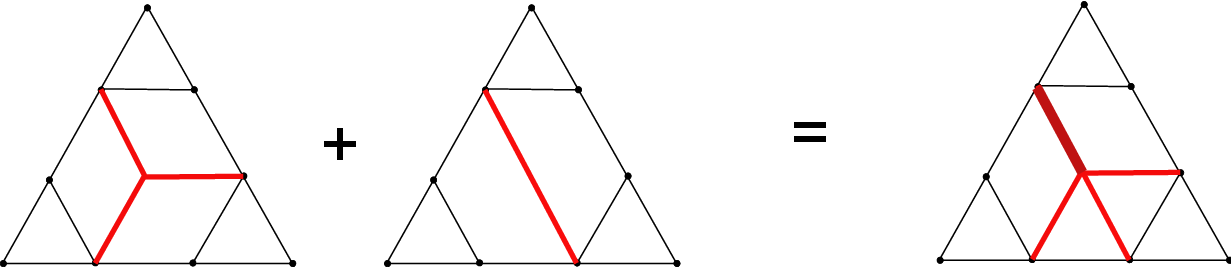}
\caption{Adding two honeycomb graphs, the thicker edge on the right hand side honeycomb indicates multiplicity.}
\label{fig:BZ3eq}
\end{figure}

Elements of $Q_{\tree}(sl_m(\C))$ are now constructed by pasting copies of the this triangle together over the structure of $\tree.$

\begin{figure}[htbp]
\centering
\includegraphics[scale = 0.5]{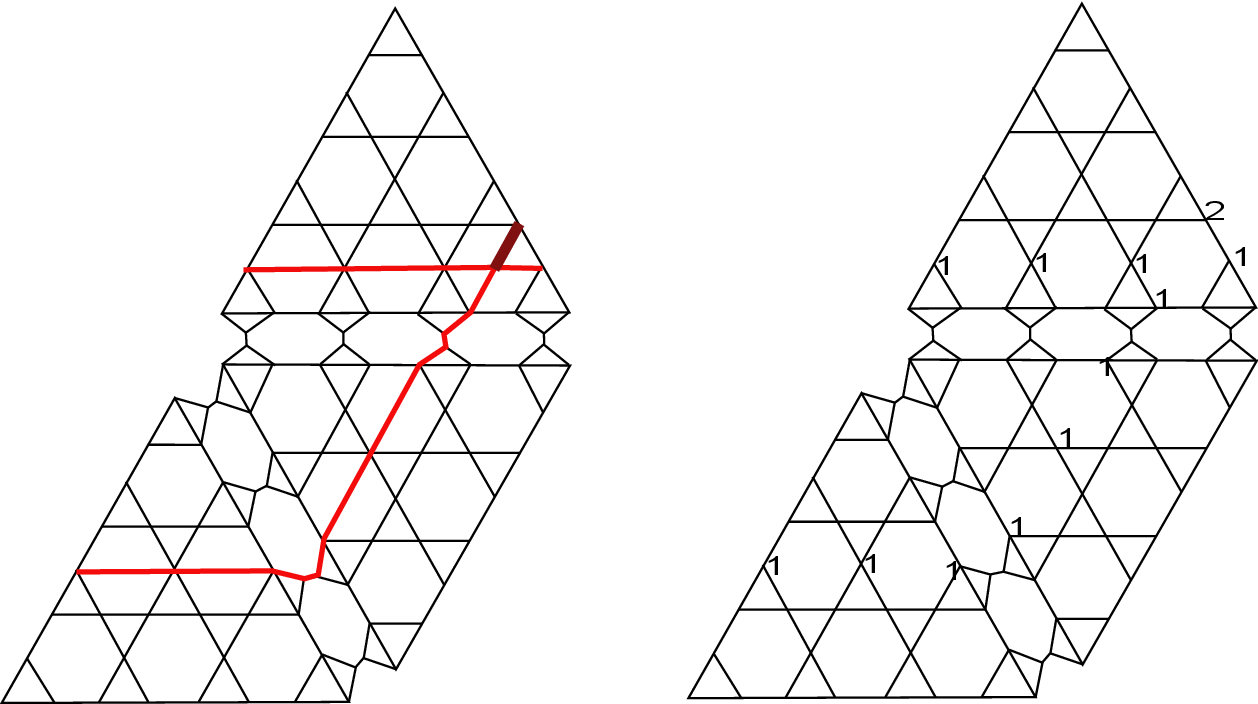}
\caption{A quilt for $sl_5(\C)$}
\label{fig:HIVEquilt}
\end{figure}

\begin{definition}
The semigroup $Q_{\tree}(sl_m(\C))$ is the set of weightings of the diagram obtained
from the tree $\tree$ such that

\begin{enumerate}
\item the weighting of each triangle is a BZ triangle,\\
\item for any pair of adjacent triangles meeting along common edges, with counter-clockwise oriented weights $(a_1, \ldots, a_{2m-2}), (A_1, \ldots, A_{2m-2}),$ we have $a_{2i-1} + a_{2i} = A_{2(m-i) -1} + A_{2(m-i)}.$
\end{enumerate}
\end{definition}

\begin{figure}[htbp]
\centering
\includegraphics[scale = 0.4]{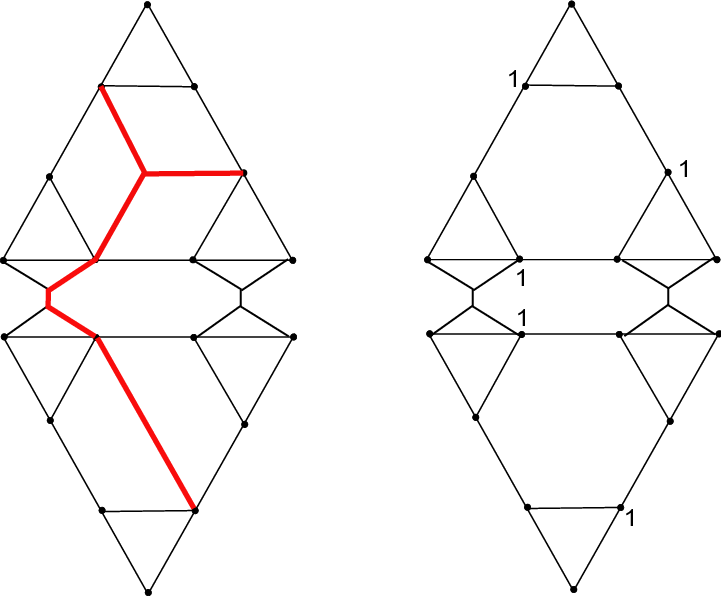}
\caption{A quilt for $sl_3(\C)$}
\label{fig:HIVEquilt2}
\end{figure}

We dualize this definition to form honeycomb graphs corresponding to elements of $Q_{\tree}(sl_m(\C)).$  Each weighting in $Q_3(sl_m(\C))$ is dualized to a honeycomb, and honeycombs which are compatible by condition $(2)$ above can be glued along the "stitching" added in between the BZ triangles, see Figures \ref{fig:HIVEquilt} and \ref{fig:HIVEquilt2}.  Notice in Figure \ref{fig:HIVEquilt2}, the left boundry entries on the top and bottom triangle both sum to $1$, and the right hand side entries both sum to $0$, ensuring that these triangles can be glued along the boundary.

\section{Generators and relations for $Q_{\tree}(sl_3(\C))$}

In this section we describe a presentation of $Q_{\tree}(sl_3(\C))$ for all $\tree$, in particular we prove Theorem \ref{main}.

\subsection{Generators}


First we describe the algebraic structure of $Q_3(sl_3(\C))$, see Figure \ref{fig:BZ3} for the generators of this semigroup.  The triangles $X$ and $Y$ each represent the determinant form in their respective triple tensor products of fundamental representations, and $P_{ij}$ represents the unqiue invariant in $V(\omega_1)\otimes V(\omega_2)$ in the positions $i, j$ respectively.

\begin{figure}[htbp]
\centering
\includegraphics[scale = 0.5]{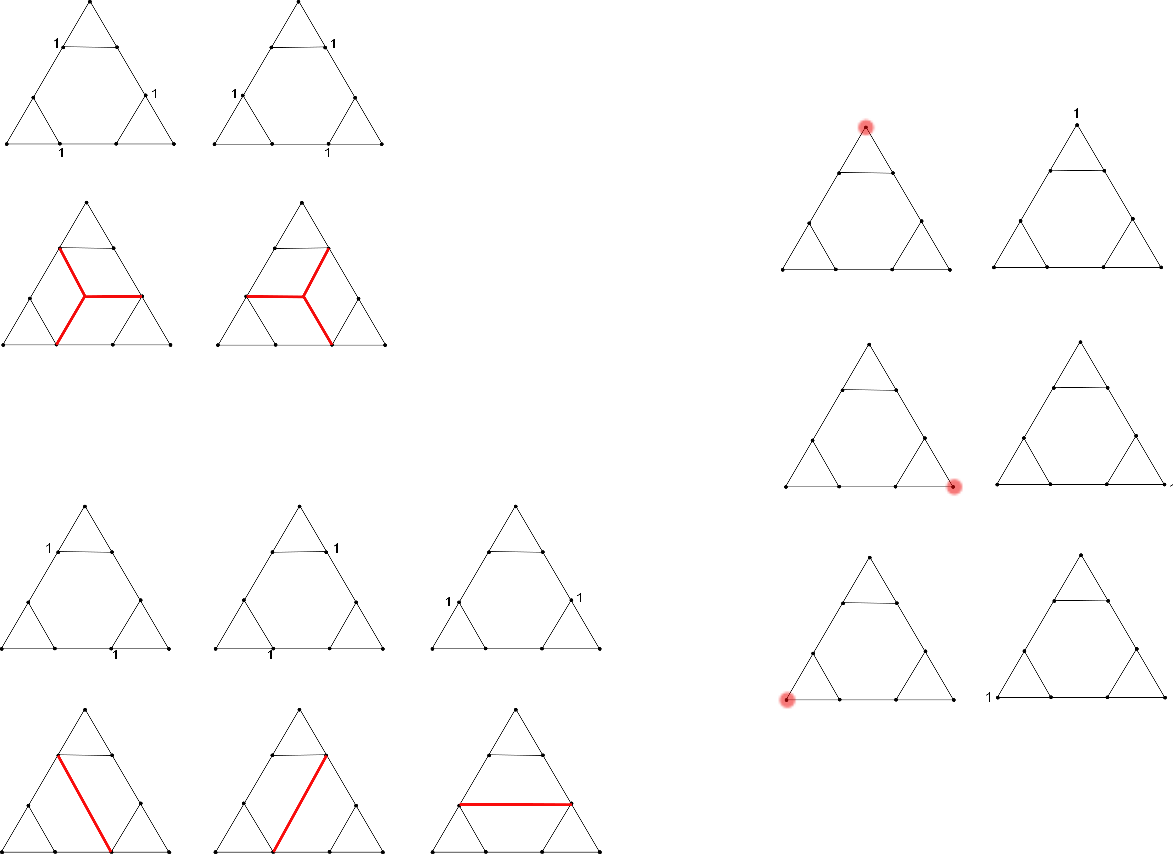}
\caption{Generators of $Q_3(sl_3(\C))$}
\label{fig:BZ3}
\end{figure}

\begin{proposition}
The semigroup $Q_3(sl_3(\C))$ is generated by $X, Y,$ and  $P_{ij}$ 
for $i, j \in \{ 1, 2, 3\}$ distinct, subject to the relation $XY = P_{12}P_{23}P_{31}.$
\end{proposition}

\begin{proof}
This is easily established with a software package which handles polyhedral cones, such as 4ti2 \cite{4ti2}. 
\end{proof}

This result says that $Q_3(sl_3(\C))$ is generated by triangles with only fundamental
weights along their edges.  As the triangle gets larger, the diagram admits increasingly chaotic graphs which cannot be decomposed, this is illustrated in Figure \ref{fig:BZcounter2}.  Generation results like Theorem \ref{main} cannot hold at larger scale either, Figure \ref{fig:BZcounter} illustrates this. Notice that the bottom BZ triangle is a generator and the top is composite, yet no factorization of the top can be extended to the bottom.  This example can be made for  arbitrary $m \geq 4.$

\begin{figure}[htbp]
\centering
\includegraphics[scale = 0.35]{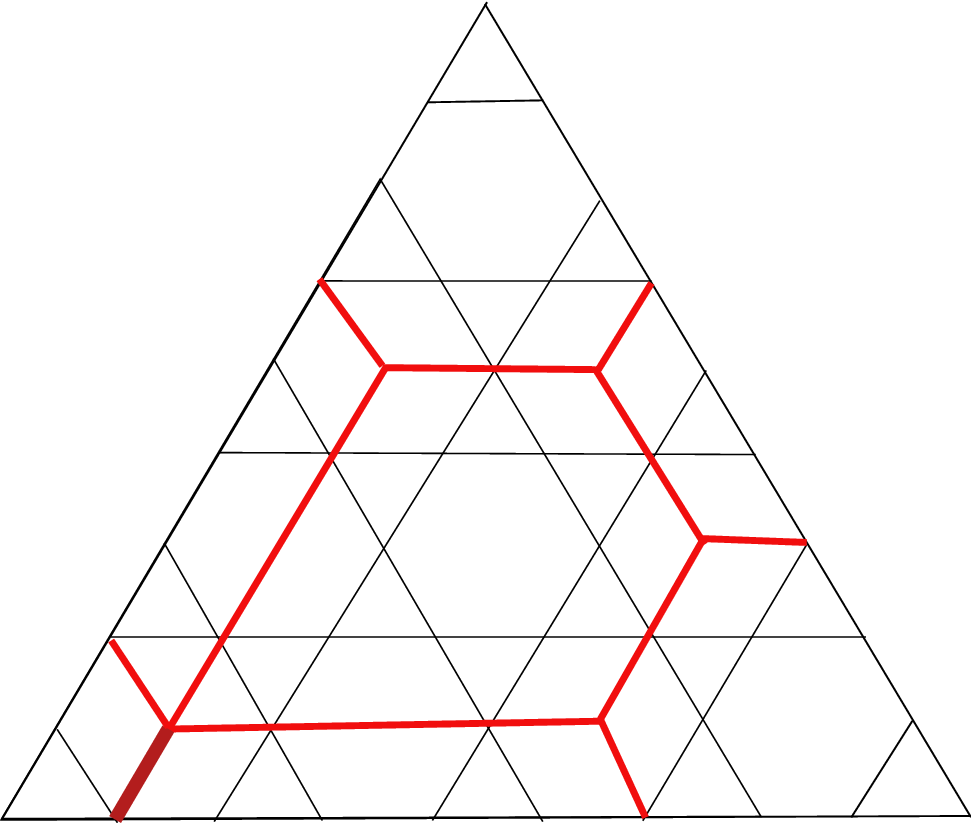}
\caption{An indecomposable element of $Q_3(sl_6(\C))$ with an edge of weight $2$}
\label{fig:BZcounter2}
\end{figure}

\begin{figure}[htbp]
\centering
\includegraphics[scale = 0.35]{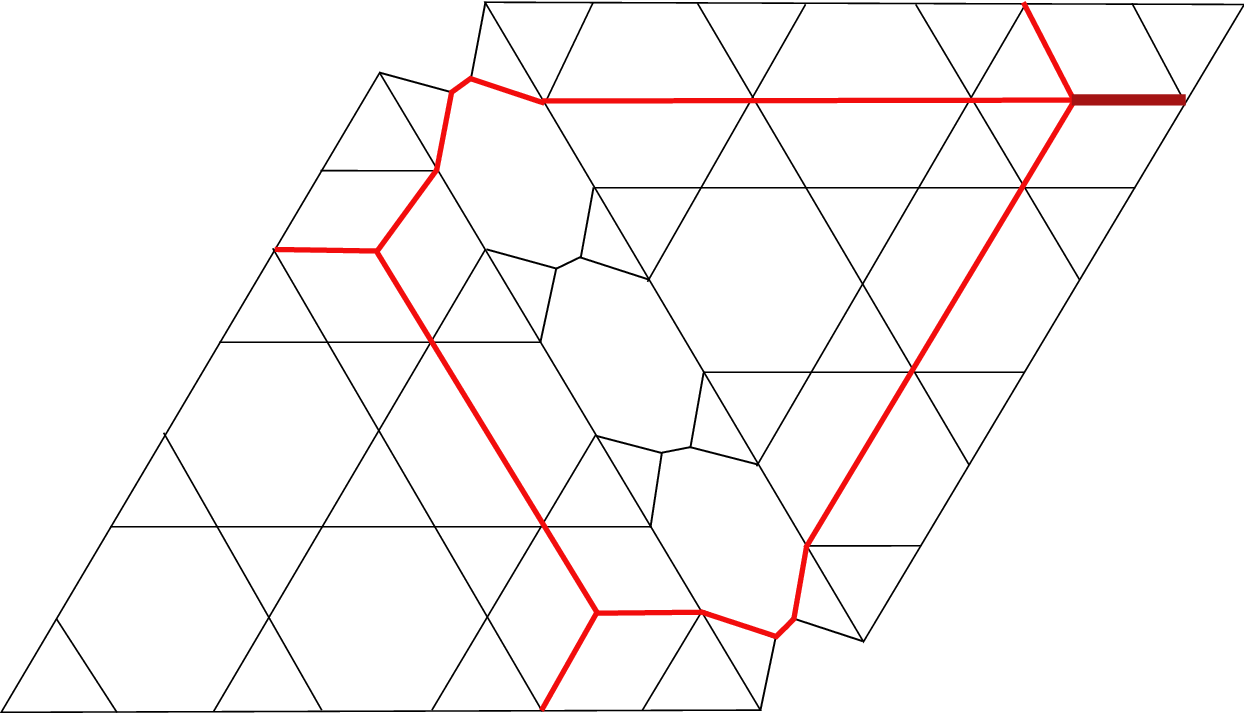}
\caption{An indecomposable quilt for $sl_5(\C)$}
\label{fig:BZcounter}
\end{figure}

Now we extend this generation result to $BZ_{\tree}(3)$ for any tree $\tree.$  The first part of Theorem \ref{main} follows from the next lemma. 

\begin{lemma}
Let $S$ be a semigroup with a map $\partial: S \to \Z_{\geq 0}^2,$ and let $\pi_1: Q_3(sl_3(\C)) \to \Z_{\geq 0}^2$ be the projection onto one of the edges.  The fiber product semigroup
$S \times_{\Z_{\geq 0}^2} Q_3(sl_3(\C))$ defined by these maps is generated by the elements $(s, w)$ such that $s$ is a generator of $S$.
\end{lemma}

\begin{proof}
Let $(x, z)$ be any element in the fiber product $S \times_{\Z_{\geq 0}^2} Q_3(sl_3(\C)).$
Let $x = s_1 + \ldots + s_k$ be a decomposition of $x \in S$, and $z = w_1 + \ldots + w_n$ 
be a decomposition of $z$ into generators of $Q_3(sl_3(\C))$.  The image vectors
$\partial(x_1), \ldots, \partial(x_n) \in \Z_{\geq 0}^2$ and $\pi_1(w_1), \ldots, \pi_1(w_n)$ must satisfy

\begin{equation}
\sum \partial(x_i) = \sum \pi_1(w_j)\\
\end{equation}

\noindent
with each $\pi_1(w_j) \in \{(0, 1), (1, 0)\}.$   Since each element of $\Z_{\geq 0}^2$ uniquely decomposes into a sum of the generators $(0, 1), (1, 0),$ we can group
the elements $w_j$ into disjoint sets such that $\partial(x_i) = \sum \pi_1(w_j(i)).$
But then it follows that $(x_i, \sum w_j(i)) \in S_{\Z_{\geq 0}^2} Q_3(sl_3(\C)).$
\end{proof}

An induction argument now shows that $Q_{\tree}(sl_3(\C))$ is generated by quilts
which have the property that the weighting obtained by any restriction to an internal triangle
is either $0$ or a generator of $Q_3(sl_3(\C))$.    We define the support of a generator $w \in Q_{\tree}(sl_3(\C))$ to be the set of edges $e \in E(\tree)$
which are assigned a non-zero weighting by $w.$.  If the support of $w$ has disjoint connected components, it follows easily that $w$ can be factored.
If $w$ has connected support, the corresponding subtree of $\tree$ must be proper, as defined in the introduction. 

\begin{lemma}
For any proper subtree $\tree' \subset \tree,$ there are precisely two generators $w \in Q_{\tree}(sl_3(\C))$ with support $\tree'$. 
\end{lemma}

\begin{proof}
We consider the case $Q_3(sL_3(\C))$, where the graph is a single trinode.  If the support is the whole trinode, then the weighting can be either $X$ or $Y$, and the weight data at a leaf determines which of these it must be. Each pair of edges $\{i, j\}$ likewise has precisely two weightings which carry $\{i, j\}$ as their support, determined by the weight at either edge.  The general lemma follows from this by induction. 
\end{proof}

The number of proper subtrees of $\tree$ is exactly $2^n - n - 1,$ as this is the number of subsets of the set of leaves, modulo singletons and the empty set.   It follows that there are $2(2^n - n - 1)$ minimal generators for each $Q_{\tree}(sl_3(\C)),$ this proves Theorem \ref{main2}.

\subsection{Relations}

We now describe a generating set for the ideal $I_{\tree}(sl_3(\C))$ that vanishes on these generators in $\C[Q_{\tree}(sl_3(\C))].$

\begin{definition}
For a vertex $v \in V(\tree)$ we let $\tree_1, \tree_2, \tree_3$ be the connected components
of the forest obtained by cutting along the edges connected to $v.$  
\end{definition}

\begin{figure}[htbp]
\centering
\includegraphics[scale = 0.4]{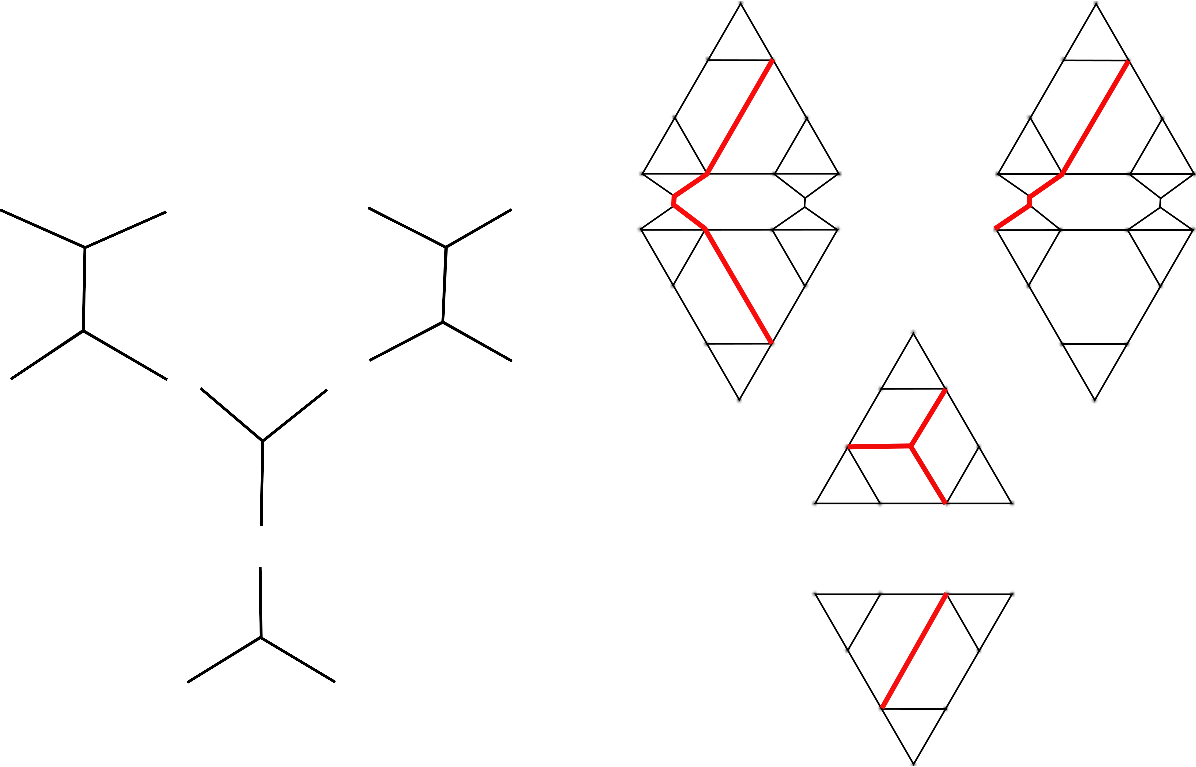}
\caption{Induced quilts on the components $\tree_1, \tree_2, \tree_3$}
\label{fig:triweight}
\end{figure}

For any weighting or graph $w \in Q_{\tree}(sl_3(\C))$ we obtain elements $w_3 \in Q_3(sl_3(\C)),$ and  $w_{\tree_i} \in Q_{\tree_i}(sl_3(\C))$ $i = 1, 2, 3$ by restriction.  Note that if $[w_1 \ldots w_k] - [v_1 \ldots v_j] \in I_{\tree}(sl_3(\C))$ then the restrictions of this word to $Q_3(sl_3(\C))$ and the $Q_{\tree_i}(sl_3(\C))$ will be in $I_3(sl_3(\C))$ and  $I_{\tree_i}(sl_3(\C))$ respectively.  We consider the relations depicted in Figures \ref{fig:BZquilteq1} and \ref{fig:BZquilteq2}.  In Figure \ref{fig:BZquilteq1},  $A, B, C, D$ are meant to represent weightings on the quilts on either side of the depicted edge.  If $A$ and $B$ are both compatible with $C$ and $D$ along this edge, a swap can be made.   Similarly, the main relation of $Q_3(sl_3(\C))$ can be extended to any $Q_{\tree}(sl_3(\C))$ at any vertex $v \in V(\tree)$ for any compatible $A, B, C, D, E, F, G, H, I$ for the appropriate quilts on $Q_{\tree_i}(sl_3(\C))$, as depicted in Figure \ref{fig:BZquilteq2}.  We now complete the proof of Theorem \ref{main}

\begin{figure}[htbp]
\centering
\includegraphics[scale = 0.37]{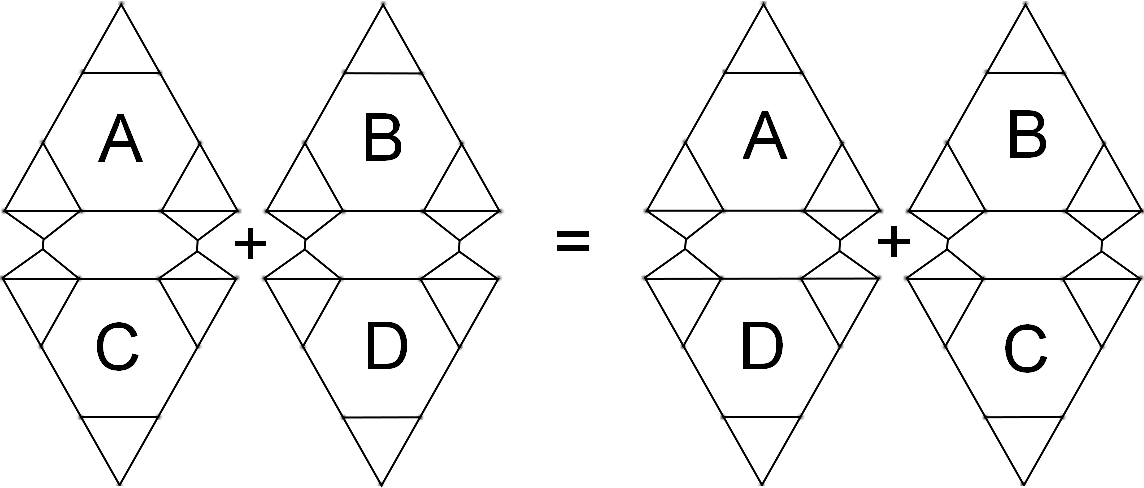}
\caption{The swap relation on an edge}
\label{fig:BZquilteq1}
\end{figure}

\begin{figure}[htbp]
\centering
\includegraphics[scale = 0.55]{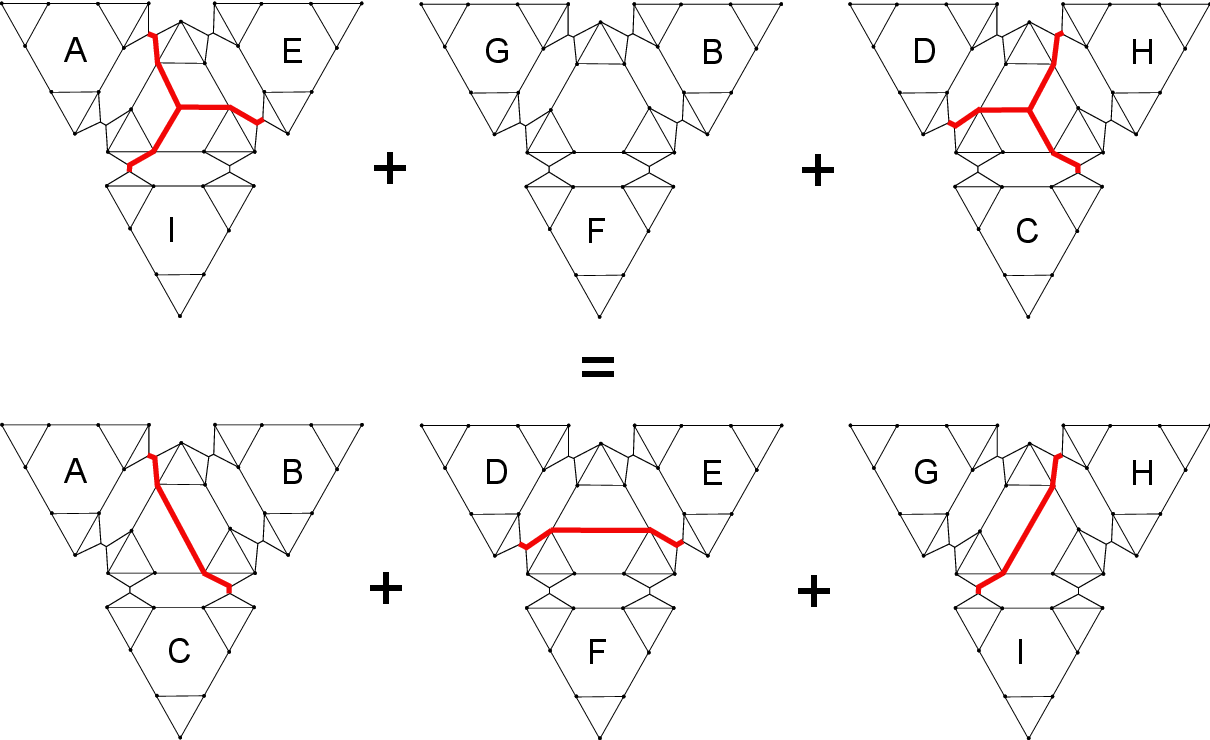}
\caption{Extension of the main relation of $Q_3(sl_3(\C))$}
\label{fig:BZquilteq2}
\end{figure}

\begin{proposition}\label{relt}
The two types of relations depicted in Figures \ref{fig:BZquilteq1} and \ref{fig:BZquilteq2}
suffice to generate $I_{\tree}(sl_3(\C)).$
\end{proposition}

\begin{proof}
We select a trinode connected to two leaves $T \subset \tree$.  We write $\tree = T \cup \tree'$ for $\tree'$ the complement of $T$.   Suppose that $c^1\ldots c^n = d^1 \ldots d^m$ is a relation
among generators $c^i, d^j \in Q_{\tree}(sl_3(\C)).$   We consider the restriction of the words
$c_T^1 \ldots c_T^n$; $d_T^1 \ldots d_T^m$ and $c_{\tree'}^1 \ldots c_{\tree'}^n$; $d_{\tree'}^1 \ldots d_{\tree'}^m$ to $T$ and $\tree'$ respectively.

 By induction, there is a sequence of modifications of the word $c_{\tree'}^1 \ldots c_{\tree'}^n$
by the above relations which result in the word $d_{\tree'}^1 \ldots d_{\tree'}^m.$  Note that each step in this sequence does not change the list of values along the edge $e$ which separates $T$ from $\tree'.$  As a consequence, each modification can be extended to a relation between weightings of $\tree$ which are trivial over $T.$  The result is a word $g_{\tree}^1 \ldots g_{\tree}^k$ which differs from $d_{\tree}^1 \ldots d_{\tree}^m$ only over the trinode $T.$  

Now, since the relation $W_1W_2W_3 = Y_1Y_2$  does not change the list of values along the edge $e,$ we may use it to transform $g_{T}^1 \ldots g_{T}^k$ to $d_{T}^1 \ldots d_{T}^m.$   Furthermore, since the list of boundary values along $e$ does not change, each application of this relation can be lifted to relations on weightings of $\tree$ which are trivial over $\tree'.$  The resulting word now differs from $d^1_{\tree}\ldots d^m_{\tree}$ only up to an application of the relation in Figure \ref{fig:BZquilteq1}. 
\end{proof}

\section{Generators and relations for $P_{\tree}$}

In this section we  prove Theorem \ref{grass} and relate a particular instance
of our quilt semigroups to a semigroup of Gel'fand-Tsetlin patterns.  Define $F: Q_{\tree}(sl_3(\C)) \to \R$ to be the function obtained by first projecting to the boundary weightings in $(\Z_{\geq 0}^2)^n,$ then summing all of the second fundamental weight components. This function is non-negative on all of $Q_{\tree}(sl_3(\C)).$ From our discussion in the introduction, the semigroup $P_{\tree}$ can be identified with the $0$ set of $F.$   It follows that $P_{\tree}$ is face of $Q_{\tree}(sl_3(\C))$.  As a result, in order to prove Theorem \ref{grass}, we need only find the generators of $Q_{\tree}(sl_3(\C))$ which lie in $P_{\tree}.$

\begin{lemma}
A minimal generator of $Q_{\tree}(sl_3(\C))$ is in $P_{\tree}$ if and only if its associated
proper tree $\tree'$ has the property that the unique path between any two leaves $\ell_1, \ell_2 \in \tree'$ contains an odd number of trivalent vertices. 
\end{lemma}

\begin{proof}
Any generator $w \in P_{\tree}$ can only have the first fundamental weight as a boundary value. Assuming this is the value of $w$ at a leaf $\ell_1,$ we see that each pair of trinodes encountered on the unique path to a second leaf $\ell_2$ switches the weighting, either from $(1, 0)$ to $(0, 1)$ or vice-versa.  It follows that exactly an odd number of such trinodes can be encountered on any such path.
\end{proof}

\begin{figure}[htbp]
\centering
\includegraphics[scale = 0.35]{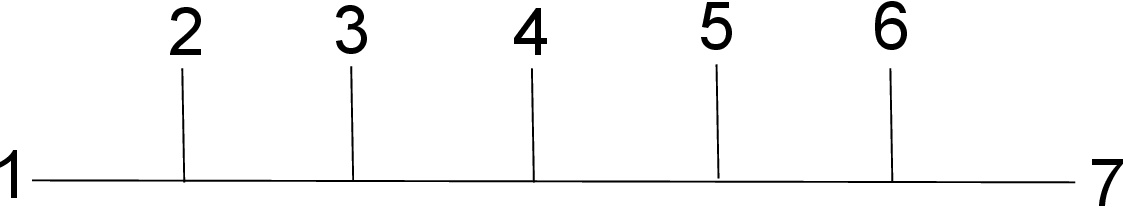}
\caption{A caterpillar graph}
\label{fig:caterpillar}
\end{figure}

We now look at a special case, $P_{\tree_0}$, where $\tree_0$ is the caterpillar tree, depicted in Figure \ref{fig:caterpillar}.  Note that all trinodes in $\tree_0$ are connected to some leaf.  It follows that the only generators in $P_{\tree_0}$ are those defined by the proper subtrees $\tree_{i, j, k}$, which contain one trinode, with three paths emanating to the leaves on the specified indices, for all $\{i, j, k\} \subset \{1, \ldots, n\}.$  

These generators are then subject to the induced relations from $Q_{\tree}(sl_3(\C)).$
Since the generator $Y \in Q_3(sl_3(\C))$ never appears at a trinode of a generator of $P_{\tree_0},$ the cubic relations we described above make no contribution.  It follows that 
the ideal $I_{\tree_0}$ that vanishes on the generators of $P_{\tree_0}$ is purely quadratic and binomial. For two generators $w_{i,j,k}$ and $w_{r, s, t}$ we have the following relations, 

\begin{equation}
w_{i, j, k}w_{r, s, t} = w_{r, j, k}w_{i, s, t}, \ \ \  i, r > j, s\\
\end{equation}

\begin{equation}
w_{i, j, k}w_{r, s, t} = w_{i, j, t}w_{r, s, k}, \ \ \  s, j > t, k\\
\end{equation}

In \cite{MSt}, Theorem $14.16$, Sturmfels and Miller construct a degeneration
of the Pl\"ucker algebra $P_{3, m}$ to a semigroup $GT_n(\omega_3)$  (our notation) of Gel'fand-Tsetlin patterns for the dominant weight $\omega_3$ of $GL_n(\C).$  These are weightings of the vertices of the diagram below by non-negative integers for some $\lambda \in \Z_{\geq 0}$ which satisfy the $interlacing$ $inequalities$ indicated by the arrows. Namely, whenever an arrow connects two entries $m \to f$, we must have $f \geq m.$

\begin{figure}[htbp]
$$
\begin{xy}
(0,0)*{\bullet} = "1";
(24,0)*{\bullet} = "2";
(48,0)*{\bullet} = "3";
(12,-15)*{\bullet} = "a";
(36,-15)*{\bullet} = "b";
(60,-15)*{\bullet} = "c";
(24,-30)*{\bullet} = "4";
(48,-30)*{\bullet} = "5";
(72,-30)*{\bullet} = "6";
(36,-45)*{\bullet} = "e";
(60,-45)*{\bullet} = "f";
(84,-45)*{\bullet} = "g";
(48,-60)*{\bullet} = "7";
(72,-60)*{\bullet} = "8";
(60,-75)*{\bullet} = "h";
"1"; "a";**\dir{-}? >* \dir{>};
"a"; "2";**\dir{-}? >* \dir{>};
"2"; "b";**\dir{-}? >* \dir{>};
"b"; "3";**\dir{-}? >* \dir{>};
"3"; "c";**\dir{-}? >* \dir{>};
"a"; "4";**\dir{-}? >* \dir{>};
"4"; "b";**\dir{-}? >* \dir{>};
"b"; "5";**\dir{-}? >* \dir{>};
"5"; "c";**\dir{-}? >* \dir{>};
"c"; "6";**\dir{-}? >* \dir{>};
"4"; "e";**\dir{-}? >* \dir{>};
"e"; "5";**\dir{-}? >* \dir{>};
"5"; "f";**\dir{-}? >* \dir{>};
"f"; "6";**\dir{-}? >* \dir{>};
"6"; "g";**\dir{-}? >* \dir{>};
"e"; "7";**\dir{-}? >* \dir{>};
"7"; "f";**\dir{-}? >* \dir{>};
"f"; "8";**\dir{-}? >* \dir{>};
"8"; "g";**\dir{-}? >* \dir{>};
"7"; "h";**\dir{-}? >* \dir{>};
"h"; "8";**\dir{-}? >* \dir{>};
(4,0)*{\lambda};
(28,0)*{\lambda};
(52,0)*{\lambda};
(16,-15)*{\lambda};
(40,-15)*{\lambda};
(64,-15)*{a_{1, 3}};
(28,-30)*{\lambda};
(52,-30)*{a_{1, 2}};
(76,-30)*{a_{2, 3}};
(40,-45)*{a_{1, 1}};
(64,-45)*{a_{2, 2}};
(88,-45)*{a_{3, 3}};
(52,-60)*{a_{2, 1}};
(76,-60)*{a_{3, 2}};
(64,-75)*{a_{3, 1}};
\end{xy}
$$\\
\caption{A Gel'fand-Tsetlin pattern in $GT_6(\omega_3)$}
\label{fig:GT}
\end{figure}
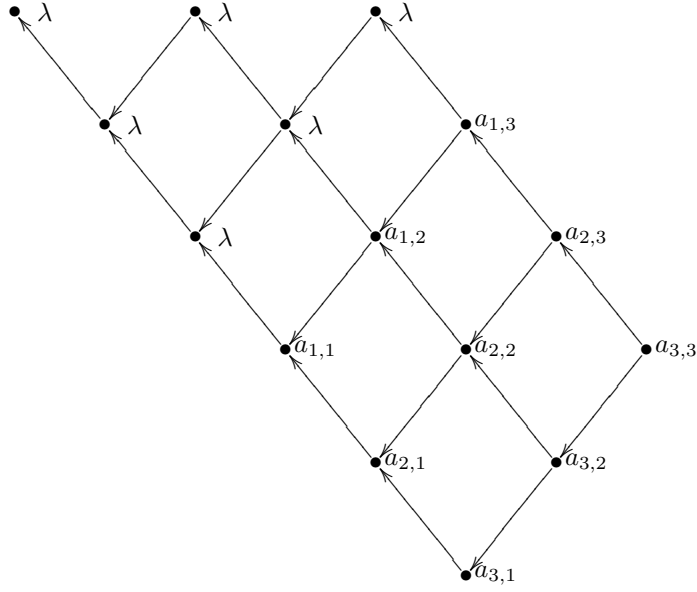

Notice the pattern is composed of three diagonals that go from top left to bottom right. 
Because the upper left triangle is forced to have all weights equal to $\lambda$ by the interlacing inequalities, effectively each of these diagonals has length $n-3,$ we label
the entries on the diagonals accordingly, starting from the top left, see Figure \ref{fig:GT}. A generator $w_{i, j, k}, i < j < k$ is mapped to the unique pattern with $\lambda = 1 = a_{i - 1,3 } = a_{j - 2 ,2 } = a_{k- 3 ,1},$ and everything "below" these entries equal to $0$.

In \cite{MSt}, the ideal $J_{3, n} \subset \C[\ldots w_{i, j, k} \ldots;\ i < j < k ]$ which cuts out the semigroup algebra $\C[GT_n(\omega_3)]$ is generated by binomials of the form 

\begin{equation}
w_{i, j, k}w_{r, s, t} - w_{(i, j, k) \wedge (r, s, t)} w_{(i, j, k) \vee (r, s, t)}\\
\end{equation}

\noindent
where $(i, j, k) \wedge (r, s, t)$ is the triple $(min(i, r), min(j, s), min(k, t))$
and $(i, j, k) \vee (r, s, t)$ is the triple $(max(i, r), max(j, s), max(k, t)).$
Note that if $w_{\sigma}w_{\tau} - w_{\sigma'}w_{\tau'}$ is either of the binomial
generators of $I_{\tree_0}$ defined above, then $\sigma \wedge \tau = \sigma' \wedge \tau'$ and $\sigma \vee \tau = \sigma' \vee \tau'.$  It follows then that $I_{\tree_0} \subset J_{3, n}$. 

On the other hand, the relation $w_{i, j, k}w_{r, s, t} - w_{(i, j, k) \wedge (r, s, t)} w_{(i, j, k) \vee (r, s, t)}$ is obtained by switching either the first, the last, or both the first and the last entries of the triple indices, and so it can be obtained from the relations defining $I_{\tree_0}.$  Therefore we must have $I_{\tree_0} = J_{3, n}$.

The generators of each semigroup $P_{\tree}$  contain a subset corresponding to triples of indices $(i, j, k)$ as the generators above.  However, if $\tree$ is not homeomorphic to $\tree_0$, it is a simple exercise to check that it must automatically contain additional minimal generators.   In this sense $P_{\tree_0}$ is the simplest member of this family of semigroups. 
We also remark that it is possible to construct an explicit bijection $P_{\tree_0} \cong GT_n(\omega_3)$, and this can also be done for the corresponding semigroups for all $sl_m(\C).$

\bigskip
\noindent
Christopher Manon:\\
cmanon@gmu.edu,\\
Department of Mathematics,\\ 
George Mason University\\ 
Fairfax, VA 22030 USA\\

\bigskip
\noindent
Zhengyuan Zhou:\\
zyzhou@stanford.edu,\\
Department of Electrical Engineering,\\ 
Stanford University\\ 
Stanford, CA 94305 USA\\

\date{\today}

\end{document}